 \documentclass[final,3p]{elsarticle22}

\usepackage{graphicx}

\usepackage{amssymb}
\usepackage{amsthm,amsfonts}

\usepackage{amsmath}

\newtheorem{theorem}{Theorem}
\newtheorem{corollary}[theorem]{Corollary}
\newtheorem{lemma}[theorem]{Lemma}

\newtheorem{proposition}[theorem]{Proposition}
\newtheorem{prop}[theorem]{Proposition}
\newtheorem{remark}[theorem]{Remark}

\newcommand{\suppp}{{\rm supp}\sp{+}}
\newcommand{\suppm}{{\rm supp}\sp{-}}

\def\0{\mbox{\bf 0}}
\def\1{\mbox{\bf 1}}

\def\phi{\varphi}
\def\epsilon{\varepsilon}

\def\R{{\mathbb{R}}}

\def\Z{{\mathbb{Z}}}

\def\P*{{\rm P}_*}

\begin{document}

\begin{frontmatter}



 \title{
An Improved  Proximity Bound for Bike-Dock Reallocation Problem\\
in Bike Sharing System
}

\author{Akiyoshi Shioura}
\ead{shioura.a.aa@m.titech.ac.jp}
\address{Department of Industrial Engineering and Economics, Institute of Science Tokyo,
Tokyo 152-8550, Japan}

\begin{abstract}
 We consider a class of nonlinear
integer programming problems arising from
re-allocation of dock-capacity in a bike sharing system.
 The main aim of this note is to derive an improved proximity bound for the problem
and its scaled variant.
 This makes it possible to refine the time bound for the 
polynomial-time proximity-scaling algorithm by Freund et al.~(2022).
\end{abstract}


\end{frontmatter}



\section{Introduction.}
\setcounter{equation}{0}

 We deal with a class of nonlinear integer programming problems 
arising from re-allocation of dock-capacity in a bike sharing system,
which we refer to as the \textit{dock re-allocation problem}.
 This problem is originally discussed by Freund, Henderson, and Shmoys \cite{FHS2022},
where 
a mathematical formulation of the dock re-allocation problem
is provided 
and various algorithms for solving the problem are proposed. 
 Moreover, Freund et al.~\cite{FHS2022}
apply the proposed algorithms to data sets from several cities in the US
to investigate the effect of the (re)allocation of dock-capacity. 
 In this paper, we focus on the algorithmic results for 
the dock re-allocation problem.

 The formulation of the dock re-allocation problem is explained below.
 Let $N = \{1, 2, \ldots, n\}$ denotes the set of bike stations.
 We use two kinds of variables $d(i)$ and $b(i)$ for each station $i \in N$.
 The total number of docks at station $i$ is given by $d(i)$ + $b(i)$,
where 
the variable $d(i)$ denotes the number of open docks (without bikes),
while $b(i)$ is the number of docks with bikes, i.e.,
each of $b(i)$ docks is equipped with one bike.
 Initially, $d(i)$ and $b(i)$ are set to $\bar{d}(i)$ and 
$\bar{b}(i)$, respectively,
and we need to re-allocate docks and bikes appropriately 
to improve user satisfaction. 

 Function $c_i: \Z^2_+ \to \R$ in variables $d(i)$ and $b(i)$
represents the expected  number of dissatisfied users at station $i$
(see \cite{FHS2022} for the precise definition of 
function $c_i$),
and it is shown to have the property of \textit{multimodularity}
(see Section~2 for the definition of a multimodular function). 
 We want to minimize 
the total expected  number of dissatisfied users
$c(d, b) \equiv \sum_{i=1}^n c_i(d(i), b(i))$ 
by re-allocation of docks and bikes.

 When re-allocating docks and bikes among stations,
we need to satisfy several conditions,
where $D = \sum_{i=1}^n \bar{d}(i)$
and $B = \sum_{i=1}^n \bar{b}(i)$:
\begin{itemize}
\item The total number of docks
should be the same, i.e.,
$\sum_{i=1}^n (d(i) + b(i))=D+B$.

\item 
 The total number of bikes cannot be increased,
i.e.,     $\sum_{i=1}^n b(i) \le B$.

\item
 The difference between the new and old allocation 
of docks should be small.
 This condition is represented as
$\| (d+b) - (\bar{d} + \bar{b})\|_1
 \le  2\gamma$ 
by using $\ell_1$-distance between
the new allocation $d+b$ and the old allocation  $\bar{d}+\bar{b}$
and some positive constant $\gamma$.

\item
The new dock allocation $d+b$ should be between given
lower and upper bounds, i.e., $\ell \le d+b \le u$.
\end{itemize}
In summary,  the dock re-allocation problem, denoted as DR, 
is formulated  as follows:
\[
 \begin{array}{l|lll}
 & \mbox{Minimize }  & 
\displaystyle c(d, b) \\
& \mbox{subject to } &    \sum_{i=1}^n (d(i) + b(i)) = D+ B,\\
& &   \sum_{i=1}^n  b(i) \le  B,\\
& & 
\| (d+b) - (\bar{d} + \bar{b})\|_1
 \le  2\gamma,\\
& & \ell \le d + b \le u,\\
& &    d, b \in \Z^n_+.
 \end{array}
\]

 In \cite{FHS2022}, 
a steepest descent (or greedy) algorithm for the problem DR
is proposed, which repeatedly updates a constant number of variables by $\pm 1$.
 It is shown by using the multimodularity of the objective function that
the algorithm finds an optimal solution in at most $\gamma$ iterations.
 Hence, the running time of the steepest descent algorithm is pseudo-polynomial.

 To speed up of the algorithm, proximity-scaling technique is
applied to the steepest descent algorithm in \cite{FHS2022}
(see, e.g., \cite{Hoch94,HochShant90,MST11,Shioura04} 
for algorithms based on proximity-scaling technique).
 Time bound of the proximity-scaling algorithm in \cite{FHS2022}
is dependent on a proximity bound for the distance between optimal solutions of 
the problem DR and its scaled version DR$(\lambda)$.
 The scaled problem DR$(\lambda)$ with integer parameter $\lambda$
is the problem obtained from DR by restricting the values of 
 variables $d(i)$ and $b(i)$ to be integer multiples of $\lambda$ for $i \in N$
(see Section 2 for a more precise definition of DR$(\lambda)$).
 It is shown that
the proximity-scaling algorithm runs in 
$O(n^2 M \log n  \log (D+B))$ time,
provided that the $\ell_1$-distance between
optimal solutions of DR and DR$(\lambda)$ is $O(M\lambda)$
for some $M>0$.
 In addition, it is shown that
the proximity bound $M$ is $O(n^7)$ 
(see Lemma EC.3 of the online appendix in \cite{FHS2022}).

\begin{prop}[{\cite{FHS2022}}]
\label{proo:FHS-proximity}
 Let $\lambda$ be a positive integer, and 
$(d,b)$ be an optimal solution of  {\rm DR($\lambda$)}.
 Then, there exists an optimal solution
$(d^*, b^*)$  of  {\rm DR} such that
$\| (d^*+b^*) - (d+b)\|_1 \le  16n^6 (n+4) \lambda$.
\end{prop}
\noindent
This implies that the running time of
the proximity-scaling algorithm is
$O(n^9 \log n  \log (D+B))$, which is quite large. 
 An alternative algorithm for the problem DR based on a different approach is
proposed by Shioura \cite{Shioura2022}, which runs in
$O(n \log n \log ((D+B)/n)\log B)$.

 The main aim of this paper is to 
derive an alternative, smaller proximity
bound for DR and DR$(\lambda)$ to 
improve the running time 
of the proximity-scaling algorithm in \cite{FHS2022}.

\begin{theorem}
\label{thm:new-prox}
 Let $\lambda$ be a positive integer, and 
$(d, b)$ be an optimal solution of  {\rm DR($\lambda$)}.
 Then, there exists an optimal solution
$(d^*, b^*)$  of  {\rm DR} such that
\[
\| (d^*+b^*) - (d+b)\|_1 \le   10 n \lambda.
\]
\end{theorem}
\noindent
 This theorem shows that the proximity bound $M$ is $O(n)$, 
which is much smaller than the bound $O(n^7)$ in \cite{FHS2022}. 
 Therefore, the estimate of the running time for 
the proximity-scaling algorithm is improved to
$O(n^3 \log n  \log (D+B))$, which is comparable to
the running time $O(n \log n \log ((D+B)/n)\log B)$
of the algorithm in \cite{Shioura2022}.

\begin{remark}\rm
 We can obtain a better time bound
$O(n^2 \log n  \log ((D+B)/n))$ time
by using Theorem \ref{thm:new-prox} and 
a modified version of the proximity-scaling algorithm in \cite{FHS2022}.
See Appendix for details.
\qed
\end{remark}

Due to the page limitation, some of the proofs for theorems and lemmas are
given in Appendix.


\section{Proximity Theorem for Problem DR}
\setcounter{equation}{0}

 In this section, we describe the statement of
the proximity result for the problem DR in more detail.
 Our proximity result refers to an alternative formulation of DR
to be described below, instead of the original formulation.
 To the end of this paper, we denote $y(X) = \sum_{i \in X}y(i)$
for a vector $y \in \R^n$ and $X \subseteq N$.

 Recall that the problem DR is formulated as
\[
 \begin{array}{l|lll}
 & \mbox{Minimize }  & 
\displaystyle c(d, b) \equiv \sum_{i=1}^n c_i(d(i), b(i))\\
& \mbox{subject to } &    d(N)+b(N) = D+ B,\\
& &   b(N) \le  B,\\
& & 
\| (d+b) - (\bar{d} + \bar{b})\|_1
 \le  2\gamma,\\
& & \ell \le d + b \le u,\\
& &    d, b \in \Z^n_+,
 \end{array}
\]
where each $c_i: \Z^2_+ \to \R$ is a \textit{multimodular} function, i.e.,
it satisfies the following inequalities by definition 
(see, e.g., \cite{FHS2022,Murota05}):
\begin{equation*}
\begin{array}{ll}
&
 c_i(\alpha+1, \beta+1) - c_i(\alpha+1, \beta)
 \ge  c_i(\alpha, \beta+1) - c_i(\alpha, \beta)
\\
& \hspace*{60mm}
 (\alpha, \beta \in \Z \mbox{ with }\alpha \ge 0,\ \beta \ge 0),
\\
& 
c_i(\alpha-1, \beta+1) - c_i(\alpha-1, \beta)  \ge  c_i(\alpha, \beta) - c_i(\alpha, \beta-1)
\\
& \hspace*{60mm}
(\alpha, \beta \in \Z \mbox{ with }\alpha \ge 1,\ \beta \ge 1),
\\
& 
c_i(\alpha+1, \beta-1) - c_i(\alpha, \beta-1)  \ge  c_i(\alpha, \beta) -
c_i(\alpha-1, \beta) 
\\
\qquad
&\hspace*{60mm}
(\alpha, \beta \in \Z \mbox{ with }\alpha \ge 1,\ \beta \ge 1).
\end{array}
\end{equation*}

 To obtain an alternative formulation for DR,
consider a relaxed problem of DR obtained 
by removing the $\ell_1$-distance constraint 
$\| (d+b) - (\bar{d} + \bar{b})\|_1 \le  2\gamma$,
and let $(d^\bullet, b^\bullet)$ be 
an optimal solution of the relaxed problem
with the minimum value of $\| (d+b) - (\bar{d} + \bar{b})\|_1$.
 Such an optimal solution can be computed
in $O(n \log n\log ((D+B)/n))$ time \cite[Theorem 6.1]{Shioura2022},
and this running time 
does not affect the running time for solving DR.

 Let 
\[
 P = \{i \in N \mid 
d^\bullet(i) + b^\bullet(i) > \bar{d}(i)+\bar{b}(i) \},
\qquad Q = N \setminus P.
\]

\begin{prop}[cf.~{\cite[Section 6.2]{Shioura2022}}]
\label{prop:l1-dist-const}
There exists an optimal solution $(d, b)$ of \emph{DR}
such that 
\begin{align}
d(i)  + b(i)
\begin{cases}
  \ge \bar{d}(i)+\bar{b}(i) & \mbox{ for }i \in P, \\
 \le \bar{d}(i)+\bar{b}(i) & \mbox{ for }i \in Q.
\end{cases}
\label{eqn:DR-PQ}
\end{align} 
 Moreover, if $(d^\bullet, b^\bullet)$ violates the $\ell_1$-distance constraint,
then, every optimal solution $(d, b)$ of \emph{DR}
satisfies the inequality $\| (d+b) - (\bar{d} + \bar{b})\|_1  \le  2\gamma$  
with equality.
\end{prop}

 If $(d^\bullet, b^\bullet)$ satisfies the $\ell_1$-distance constraint,
then it is also an optimal solution of DR.
Therefore, we may assume that $(d^\bullet, b^\bullet)$ 
violates the $\ell_1$-distance constraint.
 Proposition \ref{prop:l1-dist-const} shows that
the upper/lower-bound constraint \eqref{eqn:DR-PQ}
for the values $d(i)  + b(i)$ $(i\in N)$ can be added to the problem DR,
which is denoted as DR$'$.
 Note that under the condition \eqref{eqn:DR-PQ},
the $\ell_1$-distance constraint $\| (d+b) - (\bar{d} + \bar{b})\|_1 \le  2\gamma$
is equivalent to the linear inequalities
\[
  d(P)+b(P) \le \xi_P,\qquad
  d(Q)+b(Q) \ge \xi_Q
\]
with
\[
\xi_P = \bar d(P)+ \bar b(P) + \gamma,
\qquad
\xi_Q = \bar d(Q)+ \bar b(Q) - \gamma.
\]

 In summary, the problem DR$'$ is formulated as
\[
 \begin{array}{l|lll}
\mbox{(DR$'$)} & \mbox{Minimize }  & 
\displaystyle c(d, b)\\
& \mbox{subject to } & 
  d(N)+b(N) = D+ B,\\
& & 
 b(N) \le  B,\\
& & 
  d(P)+b(P) \le \xi_P,\   d(Q)+b(Q) \ge \xi_Q,\\
& & \ell' \le d + b \le u',\\ 
& &  d, b \in \Z^n_+,
 \end{array}
\]
where vectors $\ell'$ and $u'$ are given as
\begin{align*}
\ell'(i) & =
\begin{cases}
\max\{\ell(i),  \bar{d}(i)+\bar{b}(i)\} & \mbox{ if }i \in P, \\
\ell(i) & \mbox{ if }i \in Q,
\end{cases}
\\
u'(i)  &=
\begin{cases}
u(i) & \mbox{ if }i \in P, \\
\min\{u(i),  \bar{d}(i)+\bar{b}(i)\} & \mbox{ if }i \in Q.
\end{cases}
\end{align*}
 In the following, we mainly deal with DR$'$ instead of DR.

 To state a proximity result for the problem DR$'$,
let us consider a  ``scaled'' version of DR$'$,
denoted by DR$'$$(\lambda)$,
with parameter $\lambda$ given by a positive integer.
 The scaled problem DR$'$$(\lambda)$ is obtained 
by adding to DR the following constraint:
\[
 d(i) \equiv  \widetilde{d}(i) \bmod \lambda,\ 
b(i) \equiv  \widetilde{b}(i) \bmod \lambda\  (i \in N),
\]
where
$(\widetilde{d}, \widetilde{b})$ is an arbitrarily chosen
feasible solution of DR$'$.
 Note that this constraint can be simplified to the following:
\[
 d(i) \mbox{ and }
b(i) \mbox{ are integer multiples of } \lambda\mbox{ for }  i \in N,
\]
provided that
 there exists a feasible solution $(\widetilde{d}, \widetilde{b})$ of DR$'$
such that each component of $\widetilde{d}$ and $\widetilde{b}$
is integer multiple of $\lambda$.
 Hence, the problem DR$'$$(\lambda)$ with $\lambda = 1$ is  the same
as the original problem DR$'$.
 Note that $(\widetilde{d}, \widetilde{b})$ is a feasible solution of 
DR$'$$(\lambda)$ for all $\lambda$.

 It should be mentioned that
the scaled problem DR$'$$(\lambda)$ has the same structure
as the problem DR$'$;
in particular, the multimodularity of the objective function
is preserved by this scaling operation.
 Therefore, various properties and algorithms for DR$'$
can be applied to DR$'$$(\lambda)$ with some appropriate modification.

 In this paper, we show the following proximity bound for 
optimal solutions of DR$'$ and DR$'$$(\lambda)$.

\begin{theorem}
\label{thm:bdr-proximity}
 Let $\lambda$ be a positive integer,
and $(d,b)$ be an optimal solution of {\rm DR$'$$(\lambda)$}.
 Then, there exists an optimal solution 
$(d^*,b^*)$ of {\rm DR$'$} such that 
\[
 \|(d^*+b^*) - (d+b)\|_1 < 10 \lambda n.
\]
\end{theorem}
\noindent
 Proof of this theorem is given in Section \ref{sec:drp} (and in Appendix).
 By Theorem~\ref{thm:bdr-proximity} and
the fact that the structure of the problem DR$'$$(\lambda)$
is the same as the original problem DR$'$, 
the following
(seemingly) generalized proximity result for DR$'$ can be obtained.

\begin{corollary}
 Let $\lambda$ and  $\nu$ be positive integers,
and $(d,b)$ be an optimal solution of {\rm DR$'$$(\lambda \nu)$}.
Then, there exists some $\lambda$-optimal solution 
$(d^*,b^*)$ of {\rm DR$'$$(\lambda)$} such that 
\[
 \|(d^*+b^*) - (d+b)\|_1 < 10 \lambda \nu n.
\]
\end{corollary}
As an immediate corollary to Theorem \ref{thm:bdr-proximity},
we can also obtain
the following proximity bound with respect to $\ell_\infty$-distance.

\begin{corollary}
\label{coro:bdr-proximity}
 Let $\lambda$ be a positive integer,
and $(d,b)$ be an optimal solution of {\rm DR$'$$(\lambda)$}.
 Then, there exists an optimal solution 
$(d^*,b^*)$ of {\rm DR$'$} such that 
\[
 \|(d^*+b^*) - (d+b)\|_\infty < 10 \lambda n.
\]
\end{corollary}


\section{Proof of Proximity Theorem}
\label{sec:drp}
\setcounter{equation}{0}

\subsection{Proof of Theorem \ref{thm:bdr-proximity}}

 To prove Theorem \ref{thm:bdr-proximity},
let $(d,b)$ be an optimal solution of DR$'(\lambda)$, and
show that there exists some optimal solution  $(d^*,b^*)$ of
DR$'$ such that 
\begin{equation}
\label{eqn:prox-bound}
  \|(d^*+b^*)  - (d+b)\|_1 < 10 \lambda n.
\end{equation}
 Let $(d^*,b^*)$ be an optimal solution of {\rm DR$'$}
that minimizes the value $\|d^* - d\|_1 + \|b^* - b\|_1$.
 Also, let $x =d+b$ and $x^* = d^*+b^*$.
 It suffices to show that the inequality $ \|x  - x^*\|_1 < 10 \lambda n$ holds.
 Note that by Proposition \ref{prop:l1-dist-const}
and feasibility of $(d,b)$, it holds that
\begin{align}
x(P) \le \xi_P = x^*(P),
\qquad
x(Q) \ge \xi_Q = x^*(Q).
\label{eqn:x-x*-ineq}
\end{align}

 In the proof,
we use some useful inequalities
for the objective function $c(d,b)\equiv \sum_{i=1}^n c_i(d(i), b(i))$,
which follow from multimodularity of functions $c_i$.

\begin{proposition}[{cf.~\cite[Proposition 3.1]{Shioura2022}}]
\label{prop:c-ineq}
 Let $(d,b), (d', b') \in \Z^n_+ \times \Z^n_+$,
and put $x = d+b$, $x' = d'+b'$.
 Also, let $i,j,h,k,s \in N$ be elements such that
\begin{align*}
&i \in \suppp(d-d') \cap \suppp(x-x'),
&&
j \in \suppm(b-b') \cap \suppm(x-x'),
\\
& h \in \suppm(d-d') \cap \suppm(x-x'),
&& k \in \suppp(b-b') \cap \suppp(x-x'),\\
& s \in \suppm(d-d') \cap \suppp(b-b').
\end{align*}
Then, the following inequalities hold:
\begin{align}
 c( d,  b) + c( d',  b') 
& \ge  c(d- \chi_i + \chi_h, b) +  c(d'+ \chi_i - \chi_h,  b'),
\label{eqn:c-ineq-1}
\\
 c( d,  b) + c( d',  b') 
& \ge  c(d, b+ \chi_j - \chi_k) +  c(d',  b'- \chi_j + \chi_k),
\label{eqn:c-ineq-2}
\\
\label{eqn:c-ineq-4}
 c( d,  b) + c( d',  b') 
& \ge  c(d- \chi_i, b + \chi_j) +  c(d'+ \chi_i,  b' - \chi_j),
\\
 c( d,  b) + c( d',  b')\notag \\
& \hspace*{-20mm}
\ge  c(d- \chi_i+ \chi_s, b + \chi_j- \chi_s)
 +  c(d'+ \chi_i - \chi_s,  b' - \chi_j+\chi_s),
\label{eqn:c-ineq-6}
\\
 c( d,  b) + c( d',  b')\notag \\
& \hspace*{-20mm}
\ge  c(d- \chi_i+ \chi_h, b + \chi_j- \chi_k)
 +  c(d'+ \chi_i - \chi_h,  b' - \chi_j+\chi_k).
\label{eqn:c-ineq-5}
\end{align}
\end{proposition}

By Proposition \ref{prop:c-ineq} and 
the choice of $(d,b), (d^*, b^*)$,
we obtain the following property.
Proof is given in Section \ref{proof:lem:mono-dec}.

\begin{lemma}
\label{lem:mono-dec}
\ \\
{\rm (i)} 
Let $i \in \suppp(d-d^*) \cap \suppp(x-x^*)$,
$h \in \suppm(d-d^*) \cap \suppm(x-x^*)$,
and suppose that either $i,h \in P$ or $i,h \in Q$ holds.
Then, we have
\begin{align}
&
  c(d- \lambda'\chi_i+ \lambda' \chi_h, b) > 
c(d- (\lambda'+1)\chi_i+ (\lambda'+1) \chi_h, b)
\notag\\
& \hspace*{30mm} 
(0 \le \lambda' < \min\{d(i)-d^*(i), d^*(h)-d(h)\}).
 \label{eqn:lem:mono-dec-1}
\end{align}
{\rm (ii)} 
 Let $j \in \suppm(b-b^*) \cap \suppm(x-x^*)$,
$k \in \suppp(b-b^*) \cap \suppp(x-x^*)$,
and suppose that either $j,k \in P$ or $j,k \in Q$ holds.
Then, we have
\begin{align}
& c(d, b+ \lambda'\chi_j - \lambda' \chi_k)>
 c(d, b+ (\lambda'+1) \chi_j - (\lambda'+1) \chi_k)
\notag\\
& \hspace*{30mm} 
(0 \le \lambda' < \min\{b^*(j)-b(j), b(k)-b^*(k)\}).
 \label{eqn:lem:mono-dec-2}
\end{align}
{\rm (iii)} 
 Let $i \in \suppm(d-d^*) \cap \suppm(x-x^*)$,
$j \in \suppp(b-b^*) \cap \suppp(x-x^*)$,
and suppose that either $i,j \in P$ or $i,j \in Q$ holds.
Then, we have
\begin{align}
&
c(d + \lambda' \chi_i, b- \lambda'\chi_j)>
c(d + (\lambda'+1) \chi_i, b- (\lambda'+1) \chi_j)
\notag\\
& \hspace*{15mm} 
(0 \le \lambda' < \min\{ d^*(i)-d(i), b(j)-b^*(j), b(N)-b^*(N)\}).
 \label{eqn:lem:mono-dec-3}
\end{align}
{\rm (iv)} 
 Let $i \in \suppm(d-d^*) \cap \suppm(x-x^*)$,
$j \in \suppp(b-b^*) \cap \suppp(x-x^*)$,
$s \in \suppp(d-d^*) \cap \suppm(b-b^*)$,
and suppose that either  $i,j \in P$ or $i,j \in Q$ holds.
Then, we have
\begin{align}
&
 c(d- \lambda'\chi_s+ \lambda' \chi_i, b+ \lambda' \chi_s - \lambda' \chi_j)
\notag\\
&>
c(d- (\lambda'+1)\chi_s+ (\lambda'+1) \chi_i, 
b+ (\lambda'+1) \chi_s - (\lambda'+1) \chi_j)
\notag\\
& \hspace*{10mm} 
(0 \le \lambda' < \min\{ d^*(i)-d(i), b(j)-b^*(j), 
d(s)-d^*(s), b^*(s)-b(s)\}).
\label{eqn:lem:mono-dec-4}
\end{align}
{\rm (v)} 
 Let 
\begin{align*}
& i \in \suppp(d- d^*) \cap \suppp(x -x^*),\
 j \in \suppm(b- b^*) \cap \suppm(x -x^*),\\
& h \in \suppm(d- d^*) \cap \suppm(x -x^*),\
 k \in \suppp(b- b^*) \cap \suppp(x -x^*),
\end{align*}
and suppose that
either of {\rm (a)} $i,j \in P,\  h,k \in Q$ and {\rm (b)}
$i,j \in Q,\  h,k \in P$ holds.
 Then, we have
\begin{align}
&
c(d - \lambda'\chi_i + \lambda' \chi_h, b + \lambda'\chi_j - \lambda' \chi_k)
\notag\\
&>
c(d - (\lambda'+1) \chi_i + (\lambda'+1) \chi_h, b 
+ (\lambda'+1)\chi_j - (\lambda'+1) \chi_k)
\notag\\
& \hspace*{10mm} 
(0 \le \lambda' < 
\min\{ d(i)-d^*(i), b^*(j)-b(j), d^*(h)-d(h), b(k)-b^*(k) \}).
\label{eqn:lem:mono-dec-5}
\end{align}
\end{lemma}

\begin{figure}
\begin{center}
 \includegraphics[width=0.6\textwidth]{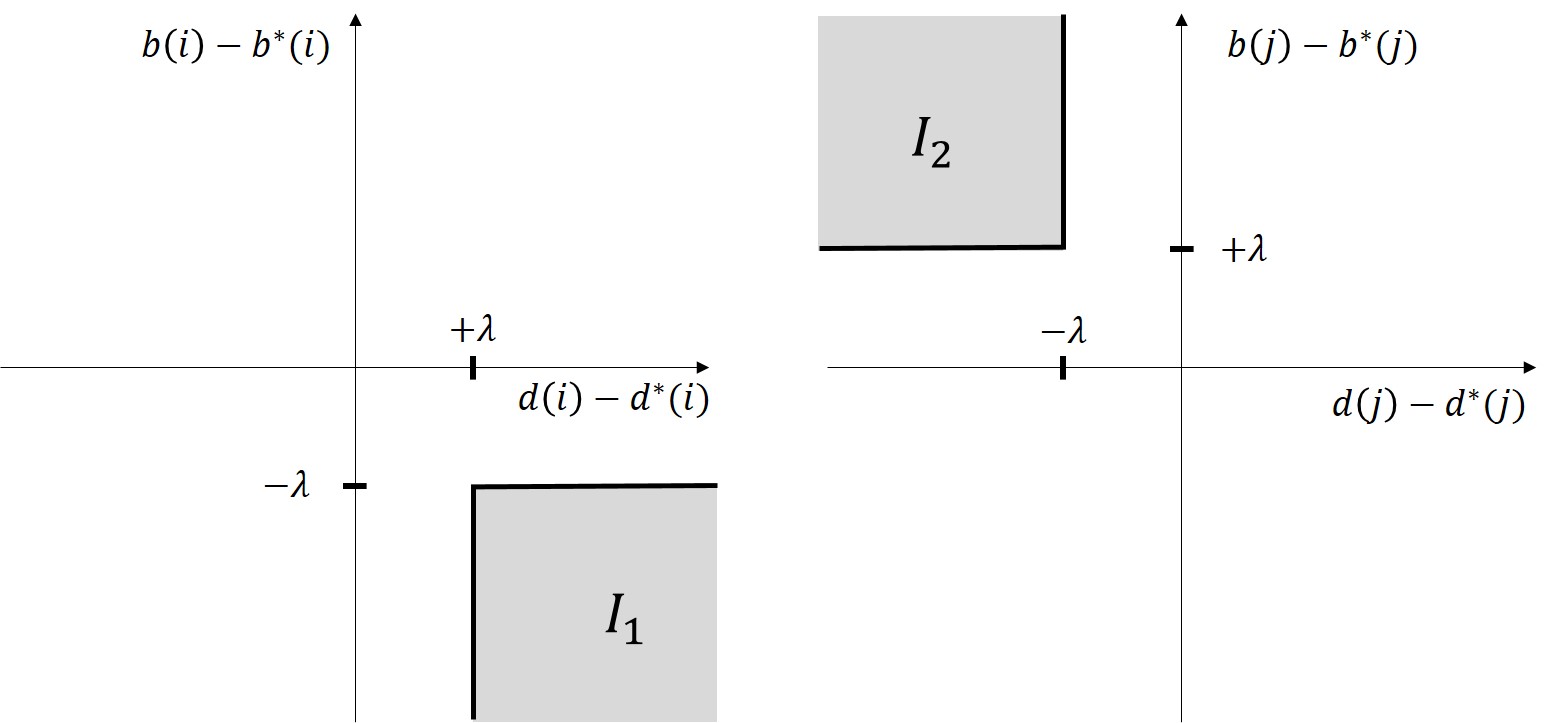}
\caption{Regions corresponding to sets $I_1$ and $I_2$\label{fig:I12}} 
\end{center}
\end{figure}
\begin{figure}
\begin{center}
{\includegraphics[width=0.6\textwidth]{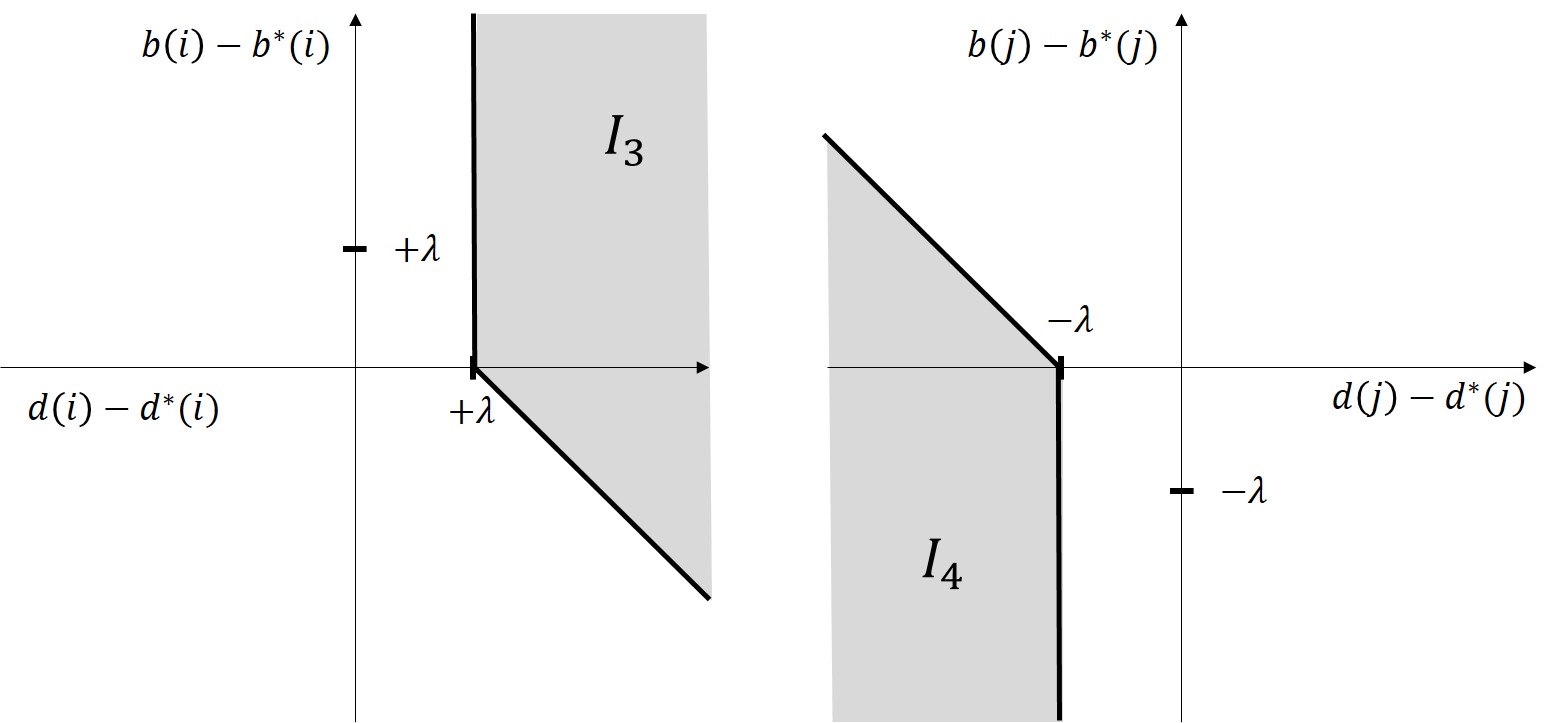}}
\caption{Regions corresponding to sets $I_3$ and $I_4$\label{fig:I34}} 
\end{center}
\end{figure}
\begin{figure}
\begin{center}
{\includegraphics[width=0.6\textwidth]{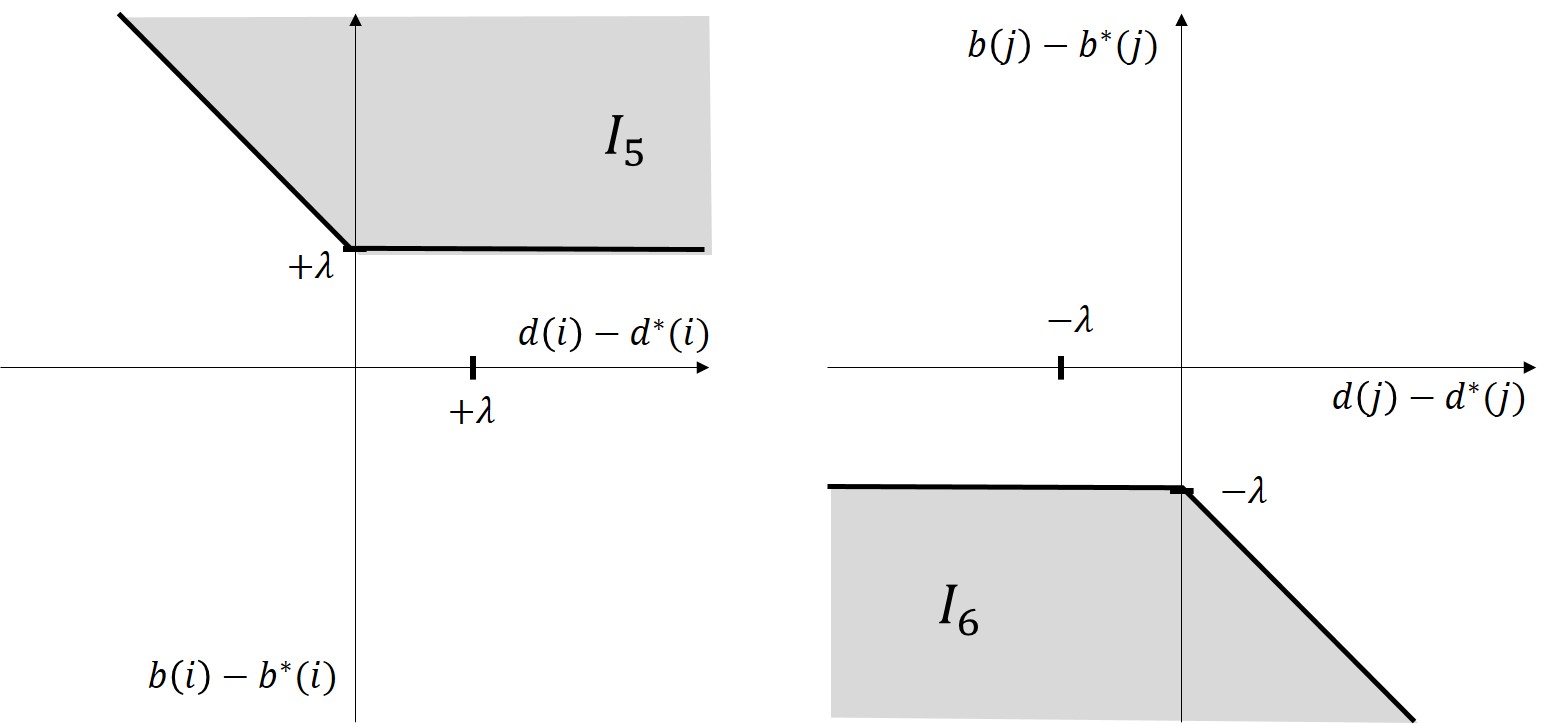}}
\caption{Regions corresponding to sets $I_5$ and $I_6$ \label{fig:I56}} 
\end{center}
\end{figure}

Derivation of the inequality $ \|x  - x^*\|_1 < 10 \lambda n$ is
done by case analysis.
 For this, define the following six sets (see Figures \ref{fig:I12}, \ref{fig:I34},
and \ref{fig:I56}):
\begin{align}
& \label{eqn:prox-cond-1} 
I_1 = \{i \in N  \mid
d(i)- d^*(i) \ge \lambda,\ b(i) - b^*(i) \le -\lambda \},
\\
& \label{eqn:prox-cond-2} 
I_2 = \{j \in N  \mid
d(j)- d^*(j) \le -\lambda,\ b(j) - b^*(j) \ge \lambda\},
\\
& \label{eqn:prox-cond-3} 
I_3 = \{i \in N  \mid
x(i)- x^*(i) \ge \lambda,\  d(i) - d^*(i) \ge \lambda\},
\\
& \label{eqn:prox-cond-4} 
I_4 = \{j \in N  \mid
x(j)- x^*(j) \le -\lambda,\  d(j) - d^*(j) \le -\lambda\},
\\
& \label{eqn:prox-cond-5} 
I_5 = \{i \in N  \mid
x(i)- x^*(i) \ge \lambda,\  b(i) - b^*(i) \ge \lambda\},
\\
& \label{eqn:prox-cond-6} 
I_6 = \{j \in N  \mid 
x(j)- x^*(j) \le -\lambda,\  b(j) - b^*(j) \le -\lambda\}.
\end{align}
 For $t=1,2,\dots, 6$, we also denote
\[
 I_t^P = I_t \cap P, \qquad  I_t^Q = I_t \cap Q,
\]
and consider the following four cases:
\begin{quote}
(Case P1)  $I_4^P = I_6^P= \emptyset$, \\
(Case P2)  $I_3^P  = I_5^P  = \emptyset$, \\
(Case P3)   
$I_3^P\ne \emptyset$, $I_6^P\ne \emptyset$,
$I_2= \emptyset$, $I_4^P= I_5^P  = \emptyset$, $b(N)-b^*(N) > - \lambda$,
and $d(N)-d^*(N) < \lambda$, 
\\
(Case P4)  
$I_4^P\ne \emptyset$, $I_5^P \ne \emptyset$,
$I_1= \emptyset$, $I_3^P= I_6^P  = \emptyset$, $b(N)-b^*(N) <  \lambda$, 
and $d(N)-d^*(N) > -\lambda$.
\end{quote}
 We will show that the cases P1, P2, P3, and P4 
cover all possible cases
by using the fact that $(d,b)$ (resp., $(d^*, b^*)$)
is an optimal solution of (DR$'(\lambda)$) (resp., DR$'$).
 Proof is given in Section \ref{sec:proof-Lemma8}.

\begin{lemma}
\label{lem:p1p2p3p4occurs}
 At least one of the cases P1, P2, P3, and P4 occurs.
\end{lemma}
\noindent

In addition, we consider additional four cases
Q1, Q2, Q3, and Q4, which are
obtained by replacing the sets $I^P_t$
 in P1, P2, P3, and P4 with $I^Q_t$,
and show that the cases Q1, Q2, Q3, and Q4 cover all possible cases;
proof is similar to that for Lemma \ref{lem:p1p2p3p4occurs} and omitted.

\begin{lemma}
\label{lem:p1p2p3p4occursQ}
 At least one of the cases Q1, Q2, Q3, and Q4 occurs.
\end{lemma}

 Hence, we have 16 possible combinations of the cases 
in $\{$P1, P2, P3, P4$\}$ and in $\{$Q1, Q2, Q3, Q4$\}$,
and we denote each of possible combinations as
P1--Q1, P1--Q2, P1--Q3, etc.

 We then show that the cases P3--Q4 and P4--Q3 are not possible, 
and derive 
the inequality $|x(i)  - x^*(i)| < 10 \lambda n$ 
for all remaining cases.

\begin{lemma}
\label{lem:bound-cases}
{\rm (i)}
The cases P3--Q4 and P4--Q3 do not occur.
\\
{\rm (ii)}
 For all combinations of the cases,
except for P3--Q4 and P4--Q3,
it holds that $|x(i)  - x^*(i)| < 10 \lambda n$ for $i \in N$.
\end{lemma}

 This concludes the proof of  Theorem \ref{thm:bdr-proximity}.

\subsection{Proof of Lemma \ref{lem:mono-dec}}
\label{proof:lem:mono-dec}

In the following we prove (v) only since other claims can 
be proved similarly and more easily.

 Setting
\[
d' =  d - \lambda'\chi_i + \lambda' \chi_h, 
\quad b' = b + \lambda'\chi_j - \lambda' \chi_k,
\quad x'=d'+b',
\]
we can rewrite the inequality \eqref{eqn:lem:mono-dec-5} as 
\begin{align}
 \label{eqn:lem:3p6p4q5q-empty:3}
&  c(d', b') > c(d'- \chi_i+  \chi_h, b' + \chi_j - \chi_k).
\end{align}

 By the choice of $\lambda'$ and $i,j,k,h$, we have
\begin{align*}
& i \in \suppp(d'- d^*) \cap \suppp(x' -x^*),\ 
 j \in \suppm(b'- b^*) \cap \suppm(x' -x^*),\\
& h \in \suppm(d'- d^*) \cap \suppm(x' -x^*),\ 
 k \in \suppp(b'- b^*) \cap \suppp(x' -x^*).
\end{align*}
 Hence, it follows from \eqref{eqn:c-ineq-5} in Proposition \ref{prop:c-ineq} that
 \begin{equation}
 \label{eqn:lem:3p6p4q5q-empty:4}
  c(d',b') + c(d^*, b^*) \ge 
 c(d'- \chi_i +  \chi_h, b' + \chi_j - \chi_k) 
+ c(d^* +  \chi_i -  \chi_h, b^* - \chi_j + \chi_k).
 \end{equation}

 Note that $(d^{**}, b^{**}) \equiv (d^* +  \chi_i -  \chi_h, b^* - \chi_j + \chi_k)$
is also a feasible solution of DR$'$ since
\begin{align*}
& (d^* +  \chi_i -  \chi_h)(P) +  (b^* - \chi_j + \chi_k)(P)
 =  d^*(P)+b^*(P), 
\\
& (d^* +  \chi_i -  \chi_h)(Q) +  (b^* - \chi_j + \chi_k)(Q)
 =  d^*(Q)+b^*(Q), 
\\
& b^{**}(N) =b^*(N) \le B.
\end{align*}
 Since $i \in \suppp(d'- d^*)$,
$j \in \suppp(b'- b^*)$,
$h \in \suppm(d'- d^*)$, and
$k \in \suppp(b'- b^*)$,
we have 
 \[
 \|d^{**} - d\|_1 + \|b^{**} - b\|_1
= (\|d^* - d\|_1 -2) + (\|b^* -b \|_1 - 2) 
 < \|d^* - d\|_1 + \|b^* -b \|_1. 
 \]
 Since $(d^*,b^*)$ is  an optimal solution of {\rm DR$'$}
that minimizes the value $\|d^* - d\|_1 + \|b^* - b\|_1$,
the inequality $c(d^*, b^*) <c(d^{**}, b^{**})$ holds.
 This inequality, together with \eqref{eqn:lem:3p6p4q5q-empty:4}, implies
the desired inequality \eqref{eqn:lem:3p6p4q5q-empty:3}.

\subsection{Proof of Lemma \ref{lem:p1p2p3p4occurs}} 
\label{sec:proof-Lemma8}

We use the following two lemmas to prove Lemma \ref{lem:p1p2p3p4occurs}.

\begin{lemma}
\label{lem:setI-2} 
 Let $S\in \{P,Q\}$. \\
{\rm (i)}
At least one of $I_3^S = \emptyset$ and  $I_4^S = \emptyset$ holds.
\\
 {\rm (ii)}
At least one of $I_5^S = \emptyset$ and $I_6^S = \emptyset$ holds.
\end{lemma}

\begin{proof}
 We consider the case $S=P$ only since the proof for the case $S=Q$ 
can be done in the same way.

 We first prove (i).
 Assume, to the contrary, that both of
 $I^P_3 \ne \emptyset$ and $I^P_4 \ne \emptyset$ hold,
and let $i \in I^P_3$ and $j \in I^P_4$.
 Then, $i \ne j$  holds since
$I^P_3 \cap  I^P_4 = \emptyset$ (see Figure \ref{fig:I34}).

 We consider the pair of vectors
 $(d - \lambda\chi_i + \lambda \chi_j, b)$.
 Since $i \in I^P_3$ and $j \in I^P_4$, we have
\begin{align*}
& (d - \lambda\chi_i + \lambda \chi_j)(P) +b(P) = d(P)  +b(P),\\
& (d - \lambda\chi_i + \lambda \chi_j)(Q) +b(Q) = d(Q)  +b(Q),\\
&  (d(i) - \lambda) + b(i) 
 = x(i) - \lambda
 \ge x^*(i) = d^*(i)+b^*(i) \ge \ell'(i),
\\
& (d(j) + \lambda)+b(j) = x(j) + \lambda 
 \le x^*(j) = d^*(j) + b^*(j) \le u'(j).
\end{align*}
 Hence, 
$(d - \lambda\chi_i + \lambda \chi_j, b)$
is  a feasible solution of DR$'(\lambda)$.
 Since $(d,b)$ is an optimal solution of DR$'(\lambda)$,
we have 
\begin{equation}
  \label{eqn:lem:setI-2-1} 
c(d,b) \le  c(d- \lambda\chi_i+ \lambda \chi_j, b).
\end{equation}
 On the other hand,
it follows from $i \in I^P_3$ and $j \in I^P_4$ that
$\min\{d(i)-d^*(i), d^*(s)-d(s)\} \ge \lambda$, which,
together with Lemma \ref{lem:mono-dec} (i),
implies that  
$c(d,b) >  c(d- \lambda\chi_i+ \lambda \chi_j, b)$,
a contradiction to \eqref{eqn:lem:setI-2-1}. 

 We omit the proof for (ii)
since the claim (ii) is symmetric to (i)
and therefore the proof  for (ii) can be obtained by interchanging
the roles of $d, d^*$ with $b, b^*$ in the proof of (i).
\end{proof}

\begin{lemma}
\label{lem:setI-4} 
 Let $S\in \{P,Q\}$. \\
{\rm (i)}
If $I_1 \ne \emptyset$, then
at least one of $I^S_4$ and $I^S_5$ is an empty set.
\\
{\rm (ii)}
If $b(N) - b^*(N) \ge  \lambda$, then
at least one of $I^S_4$ and $I^S_5$ is an empty set.
\\
{\rm (iii)}
If $I_2 \ne \emptyset$, then
at least one of  $I^S_3$ and $I^S_6$ is an empty set.
\\
{\rm (iv)}
If $b(N) - b^*(N) \le  -\lambda$, then
at least one of $I^S_3$ and $I^S_6$ is an empty set.
\end{lemma}

\begin{proof}
Proof is similar to that for Lemma \ref{lem:setI-2}
and given in \ref{sec:apdx:1}.
\end{proof}

We now prove Lemma \ref{lem:p1p2p3p4occurs}, i.e.,
we show that
at least one of Cases P1, P2, P3, and P4 occurs.
 Suppose that neither of Cases P1 and P2 occurs.
 Then, it holds that
\[
\mbox{$I^P_4 \ne \emptyset$ or $I^P_6 \ne \emptyset$ (or both),
\qquad
and \qquad
$I^P_3 \ne \emptyset$ or $I^P_5 \ne \emptyset$ (or both)}.
\]
 By Lemma \ref{lem:setI-2}, we have either 
\[
     \mbox{$I^P_4  = I^P_5  = \emptyset$, $I^P_3  \ne \emptyset,\  I^P_6  \ne \emptyset$, 
 \quad or \quad
 $I^P_3   = I^P_6  = \emptyset$,
 $I^P_4  \ne \emptyset,\  I^P_5  \ne \emptyset$.}
\]
 In the former case,
we have $I_2 = \emptyset$ by Lemma \ref{lem:setI-4} (iii)
and $b(N) - b^*(N) >  -\lambda$ by Lemma \ref{lem:setI-4} (iv);
 the inequality $b(N) - b^*(N) >  -\lambda$
and the equation $x(N) = D+B = x^*(N)$ imply
$d(N)-d^*(N) < \lambda$.
 Thus, Case P3 occurs in the former case.
 Similarly, the latter case implies Case P4.

\subsection{Proof of Lemma \ref{lem:bound-cases} (i)}
\label{sec:apdx:2}

 Lemma \ref{lem:bound-cases} (i)
follows immediately from the following fact.

\begin{lemma}
\label{lem:3p6p4q5q-empty}
\ \\
{\rm (i)}
At least one of the sets $I_3^P$, $I_6^P$, $I_4^Q$, and  $I_5^Q$
is an empty set.
\\
{\rm (ii)}
At least one of the sets $I_4^P$, $I_5^P$, $I_3^Q$, and  $I_6^Q$
is an empty set.
\end{lemma}

\begin{proof}
 We prove the statement (i) only since (ii) can be proved similarly.
 Assume, to the contrary, that
all of the sets $I_3^P$, $I_6^P$, $I_4^Q$, and  $I_5^Q$
are nonempty, and let
$i \in I_3^P$, $j \in I_6^P$,
$h \in I_4^Q$, and  $k \in I_5^Q$.
 Then, elements $i,j,h,k$ are distinct 
since $P \cap Q = \emptyset$, $I_3 \cap I_6 = \emptyset$,
and $I_4 \cap I_5 = \emptyset$.

 Consider the pair of vectors
$(\tilde{d}, \tilde{b}) \equiv 
(d - \lambda\chi_i + \lambda \chi_h, b + \lambda\chi_j - \lambda \chi_k)$.
 By the choice of $i, j, h, k$, we have
\begin{align*}
&\ell \le \tilde{d} + \tilde{b} \le u,\\
& \tilde{d}(P) + \tilde{b}(P)
 = d(P) + b(P),\\
& \tilde{d}(Q) + \tilde{b}(Q)
 = d(Q) + b(Q),\\
& \tilde{b}(N) = b(N) \le B.
\end{align*}
 That is, $(\tilde{d}, \tilde{b})$ is a feasible solution
of DR$'(\lambda)$.
 Since $(d,b)$ is an optimal solution of DR$'(\lambda)$,
 we have
\begin{equation}
c(d,b) \le c(d - \lambda\chi_i + \lambda \chi_h, b + \lambda\chi_j - \lambda \chi_k).
\label{eqn:lem:3p6p4q5q-empty:1}
\end{equation}

 On the other hand, the choice of $i,j,k,h$ that
\[
 \min\{ d(i)-d^*(i), b^*(j)-b(j), d^*(h)-d(h), b(k)-b^*(k) \} \ge \lambda,
\]
which, together with Lemma \ref{lem:mono-dec} (v),
implies that  
\[
c(d,b) >
 c(d - \lambda\chi_s + \lambda \chi_i, b + \lambda\chi_s - \lambda \chi_j),
\]
a contradiction to \eqref{eqn:lem:3p6p4q5q-empty:1}.
\end{proof}

\subsection{Proof of Lemma \ref{lem:bound-cases} (ii)}

 We prove Lemma \ref{lem:bound-cases} (ii) by case analysis.
 See Table \ref{tab:proof} for summary of the bounds for all
possible cases.
\begin{table}
\caption{Summary of the bounds for all possible cases.\label{tab:proof}}
\begin{center}
{
 \begin{tabular}{|c|c|c|c|c|}
 \hline
 & P1 & P2 & P3 & P4\\
 \hline
 Q1 & $4\lambda n$ (Lem.~\ref{lem:N1N2N3N4}) & $4\lambda n$ (Lem.~\ref{lem:p1-q2-bound}) 
 & $10\lambda n$ (Lem.~\ref{lem:p3-q1-bound})  & 
$10\lambda n$ (Lem.~\ref{lem:p3-q1-bound})\\
 \hline
 Q2 & $4\lambda n$ (Lem.~\ref{lem:p1-q2-bound}) & $4\lambda n$
	  (Lem.~\ref{lem:N1N2N3N4}) 
 & $10\lambda n$ (Lem.~\ref{lem:p3-q1-bound})  & 
$10\lambda n$ (Lem.~\ref{lem:p3-q1-bound})\\
 \hline
 Q3 & $10\lambda n$ (Lem.~\ref{lem:p1-q3-bound}) & $10\lambda n$
	  (Lem.~\ref{lem:p1-q3-bound})  &
 $8\lambda n$ (Lem.~\ref{lem:N1N2N3N4}) & --- \\
 \hline
 Q4 & $10\lambda n$ (Lem.~\ref{lem:p1-q3-bound})  & 
$10\lambda n$ (Lem.~\ref{lem:p1-q3-bound}) 
& --- & $8\lambda n$ (Lem.~\ref{lem:N1N2N3N4})\\
 \hline
 \end{tabular}
}
\end{center}{}
\end{table}


 We first consider the cases P1--Q1, P2--Q2,
P3--Q3, and P4--Q4, which imply the cases N1, N2, N3, and N4,
respectively, where

\begin{quote}
(Case N1)  $I_4 = I_6= \emptyset$, \\
(Case N2)  $I_3  = I_5  = \emptyset$, \\
(Case N3)   
$I_2= I_4= I_5  = \emptyset$, $b(N)-b^*(N) > - \lambda$,
and $d(N)-d^*(N) < \lambda$, \\
(Case N4)  
$I_1= I_3= I_6  = \emptyset$, $b(N)-b^*(N) <  \lambda$, 
and $d(N)-d^*(N) > -\lambda$.
\end{quote}

\noindent
 The following bounds are known to these cases,
which immediately yields  the bounds for the cases P1--Q1, P2--Q2,
P3--Q3, and P4--Q4:

\begin{lemma}[{\cite[Section 6.4]{Shioura2022}}]
\label{lem:N1N2N3N4}
 We have $\| x-x^* \|_1 < 4 \lambda n$
if Case N1 or N2 occurs, and
$\| x-x^* \|_1 < 8 \lambda n$ if N3 or N4 occurs.
\end{lemma}


We then consider the cases P1--Q2 and P2--Q1.

\begin{lemma}
\label{lem:p1-q2-bound}
 We have $\| x-x^* \|_1 < 4 \lambda n$
if the case P1--Q2 or  P2--Q1 occurs.
\end{lemma}

\begin{proof}
 Since the cases P1--Q2 and P2--Q1 are symmetric, 
it suffices to consider the case P1--Q2 only.
 Below we will prove the inequalities
\begin{align}
& \sum_{i \in P} |x(i)-x^*(i)| < 4\lambda |P|,
\label{eqn:lem:p1-q2-bound:1}
\\
& \sum_{i \in Q} |x(i)-x^*(i)| < 4\lambda |Q|,
\label{eqn:lem:p1-q2-bound:2}
\end{align}
which can be done in the same way as in the proof for the cases N1 and N2
in \cite{Shioura2022}.
 From \eqref{eqn:lem:p1-q2-bound:1} and \eqref{eqn:lem:p1-q2-bound:2}
we obtain the desired inequality as follows:
\[
 \| x-x^* \|_1 = 
\sum_{i \in P} |x(i)-x^*(i)| + \sum_{i \in Q} |x(i)-x^*(i)| 
< 4\lambda |P| + 4\lambda |Q| = 4\lambda |N| = 4\lambda n.
\]

\begin{figure}[t]
\begin{center}
 {\includegraphics[width=0.3\textwidth]{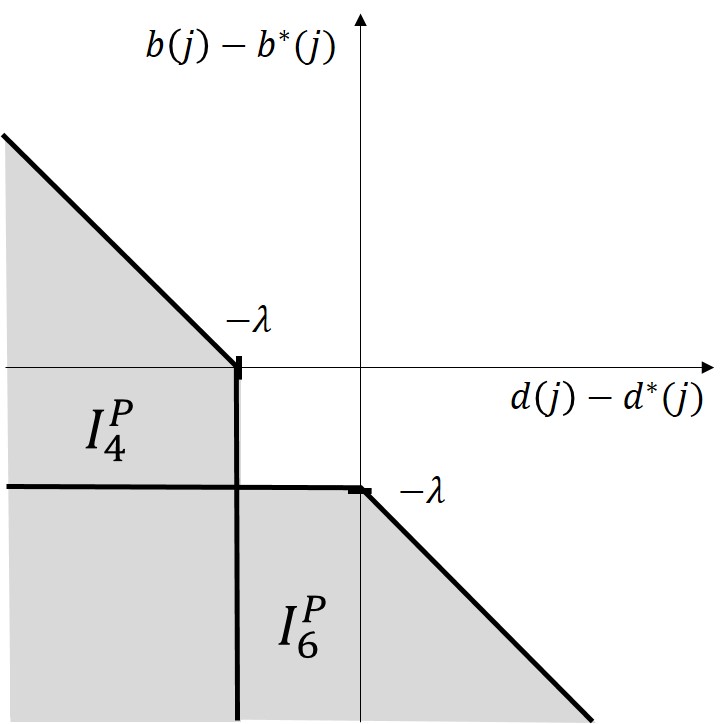}}
\caption{Proof of Lemma \ref{lem:p1-q2-bound}.
$I^P_4= I^P_6  = \emptyset$ implies 
that the point $(d(i) - d^*(i), b(i) - b^*(i))$ for $i \in P$
is not located on the shaded regions, implying that
$x(i) - x^*(i) > - 2 \lambda$. 
\label{fig:i4i6}} 
\end{center}
\end{figure}

 We first prove \eqref{eqn:lem:p1-q2-bound:1}.
 Putting $P^- = \suppm(x-x^*) \cap P$ we have
\begin{align}
\sum_{i \in P} |x(i)-x^*(i)|
& = \sum_{i \in P \setminus P^-}(x(i)-x^*(i))-\sum_{j \in P^-}(x(i)-x^*(i)) 
\notag\\
& = \Big( x(P \setminus P^-) - x^*(P \setminus P^-)\Big)  
- \Big( x(P^-) - x^*(P^-)\Big) 
\notag\\
&
 =  (x(P)-x^*(P))  -2 ( x(P^-) - x^*(P^-))\notag\\
& \le -2 ( x(P^-) - x^*(P^-)),
\label{eqn:i4i6-2}
\end{align}
where the last inequality is by 
\eqref{eqn:x-x*-ineq}.
 Since $I^P_4= I^P_6  = \emptyset$,
we have  (see Figure~\ref{fig:i4i6})
\begin{equation}
\label{eqn:i4i6-1}
  x(i) - x^*(i) > - 2\lambda  \qquad (i \in P),
\end{equation}
which, together with \eqref{eqn:i4i6-2}, implies that
\begin{align*}
\sum_{i \in P} |x(i)-x^*(i)|
&  =   -2 ( x(P^-) - x^*(P^-))
 < 4\lambda |P^-| \le 4\lambda |P|. 
\end{align*}

 The proof of \eqref{eqn:lem:p1-q2-bound:2} can be done in a similar way as
that for \eqref{eqn:lem:p1-q2-bound:1} by using
the inequality $x(Q) \ge x^*(Q)$.
\end{proof}

Proof for the remaining cases is given in \ref{sec:apdx:3}.

\section*{Acknowledgments.}

This work was supported by 
JSPS KAKENHI Grant Number 23K10995.
 The author thanks Kazuo Murota 
for valuable comments on the manuscript.

\newpage

\appendix

\section{Algorithm}

We present a proximity-scaling algorithm for computing an optimal solution
of problem DR (problem DR$'$, to be more precise).
 Our algorithm is similar to the one proposed by Freund et al.~\cite{FHS2022},
but the algorithm description and its analysis given below 
are simpler.

The algorithm consists of several $\lambda$-scaling phases,
where $\lambda$ is a scaling parameter given by a positive integer.
 Initially, $\lambda$ is set to $2^{{\lceil \log_2 ((D+B)/n)\rceil}}$,
and after each scaling phase terminates,
$\lambda$ is divided by two; 
the algorithm stops at the end onf the 1-scaling phase.
 That is, the algorithm consists of $\lceil \log_2 ((D+B)/n)\rceil$ scaling phases.

 In each $\lambda$-scaling phase, we compute an optimal solution
DR$'(\lambda)$ by using 
an optimal solution of DR$'(2\lambda)$ that is computed
in the previous scaling phase.
 We will show that the running time for each scaling phase
is $O(n^2 \log n)$, which results in
the total time complexity $O(n^2 \log n  \log ((D+B)/n))$ 
of the proximity-scaling algorithm.

 To the end of this section, we show that each $\lambda$-scaling phase
can be done in $O(n^2 \log n)$ time.
 Since the problem structure of DR$'(\lambda)$ (resp., DR$'(2\lambda)$)
is the same as that for DR$'(1)$ (resp., DR$'(2)$),
we assume in the following that $\lambda = 1$, i.e.,
 we show that an optimal solution
of DR$'(1)=$DR$'$ can be computed in $O(n^2 \log n)$ time
by using an optimal solution of DR$'(2)$.

 Computation of an optimal solution of DR$'$ consists of the following steps:
\begin{description}
 \item[Step 1]
 Let $(d^{(2)}, b^{(2)})$ be an optimal solution of DR$'(2)$.
Convert $(d^{(2)}, b^{(2)})$ of DR$'(2)$ into
a bike-optimal solution; definition of bike-optimality is given below.

 \item[Step 2] 
We reduce the feasible region of DR$'$
by using $(d^{(2)}, b^{(2)})$.

 \item[Step 3] 
Using $(d^{(2)}, b^{(2)})$ as an initial solution,
compute a feasible solution $(\hat{d}, \hat{b})$  of DR$'$
with the minimum value of 
$(1/2) \|(d+b) - (\bar{d}, \bar{b})\|_1$;
if there are many such feasible solutions, then we take the one 
minimizing the objective function value $c(\hat{d}, \hat{b})$.

 \item[Step 4] 
Using $(\hat{d}, \hat{b})$ as an initial solution,
compute an optimal solution of DR$'$ by a steepest descent algorithm
in \cite{Shioura2022}.
\end{description}
Details of each step will be explained below.

 We first convert the optimal solution $(d^{(2)}, b^{(2)})$ 
of DR$'(2)$ to a bike-optimal solution.
 A feasible solution $(d,b)$ of DR$'$ is said to be \textit{bike-optimal} 
if $c(d,b) \le c(d', b')$ holds for every feasible solution $(d',b')$
of DR$'$ with $d'+b'=d+b$.
 Using the fact that $(d^{(2)}, b^{(2)})$ is an optimal solution of DR$'(2)$,
conversion into a bike optimal solution 
can be done in $O(n \log n)$ time
(see Section 6.3.2 in \cite{Shioura2022}).
 Hence, we may assume that $(d^{(2)}, b^{(2)})$ is bike-optimal.

 We then reduce the feasible region of DR$'$
so that an optimal solution of DR$'$ can be computed efficiently.
 Since $(d^{(2)}, b^{(2)})$ is an optimal solution of DR$'(2)$,
Corollary \ref{coro:bdr-proximity} implies that there exists an optimal solution 
$(d^*,b^*)$ of {\rm DR$'$} such that 
\[
 \|(d^*+b^*)  - (d^{(2)}+b^{(2)})\|_\infty < 20  n.
\]
 Based on this fact, we update the upper/lower bound vector
$u', \ell'$ in the constraint
$\ell' \le d + b \le  u'$ by
\begin{align*}
&  u'(i) :=  \min\{u'(i), d^{(2)}(i)+b^{(2)}(i)+20  n \}
\quad (i \in N),
\\
& \ell'(i) :=  \max\{\ell'(i), d^{(2)}(i)+b^{(2)}(i)-20  n \}
\quad (i \in N).
\end{align*}
After this update, 
every feasible solution   $(d,b)$ of problem DR$'$ satisfies 
\begin{equation}
\label{eqn:gap-u'l'}
 \|(d+b)  - (d^{(2)}+b^{(2)})\|_\infty = O(n).
\end{equation}

 Let $\gamma_{\rm min}$ be an integer given as
\[
 \gamma_{\rm min} = 
\min\{(1/2) \|(d+b) - (\bar{d}+ \bar{b})\|_1
\mid (d,b): \mbox{ feasible solution of DR$'$}\}.
\]
 For each $\gamma'$ with $\gamma_{\rm min} \le \gamma' \le \gamma$,
we define an auxiliary problem DR$'_{\gamma'}$
as the problem obtained from DR$'$ by adding an extra constraint
that 
\[
 \|(d+b) - (\bar{d}+ \bar{b})\|_1 =2 \gamma'.
\]
A feasible solution $(d,b)$ of DR$'$
satisfies  $ \|(d+b) - (\bar{d}, \bar{b})\|_1 =2 \gamma'$
if and only if $(d,b)$ satisfies
\begin{align*}
  d(P)+b(P) &= \bar d(P)+ \bar b(P) + \gamma',\\
   d(Q)+b(Q)& =  \bar d(Q)+ \bar b(Q) - \gamma'.
\end{align*}

By Proposition \ref{prop:l1-dist-const},
an optimal solution of the auxiliary problem DR$'_{\gamma}$
is also optimal to DR$'$,  and vice versa.
 Based on this fact, we compute an optimal solution of DR$'$ by using 
the steepest descent algorithm in \cite{Shioura2022},
which requires an optimal solution of DR$'_{\gamma_{\rm min}}$ 
as an initial solution.
 For this purpose, we compute an optimal solution 
$(\hat{d}, \hat{b})$ of DR$'_{\gamma_{\rm min}}$. 

 Due to the definition of $\gamma_{\rm min}$,
the problem DR$'_{\gamma_{\rm min}}$ can be reformulated
as a relaxed problem of DR$'$.
 Consider the problem obtained from DR$'$ by 
removing the $\ell_1$-distance constraint and
adding an extra term 
$\Upsilon \|(d+b) - (\bar{d} + \bar{b})\|_1$
to the objective function
with a sufficiently large positive $\Upsilon$.
 The resulting objective function is given as
\begin{align*}
&  c(d,b) + \|(d+b) - (\bar{d} + \bar{b})\|_1\\
& = \sum_{i \in N}\left[
c_i(d(i),b(i)) + |(d(i)+b(i)) - (\bar{d}(i) + \bar{b}(i))|
\right],
\end{align*}
and therefore the multimodularity of objective function is preserved.
 Since the reformulated problem does not have
$\ell_1$-distance constraint, 
the result in \cite[Section 6.1]{Shioura2022} can be applied
to show that
an optimal solution of the problem DR$'_{\gamma_{\rm min}}$ is obtained
in $O(n +\tau \log n)$ time with
\begin{align}
  {\tau}  = (1/2)\min\{\|(d+b) - (d^{(2)}+b^{(2)})\|_1& \mid 
\notag\\
& 
\hspace*{-20mm}
(d,b) 
\mbox{ is an optimal solution of DR$'_{\gamma_{\rm min}}$}
\}.
 \label{eqn:def-tau}
\end{align}
 By   \eqref{eqn:gap-u'l'}, we have $\tau  = O(n^2)$, 
implying that
an optimal solution $(\hat{d}, \hat{b})$  
of the problem DR$'_{\gamma_{\rm min}}$ can be 
obtained in $O(n^2 \log n)$ time.

We can compute an optimal solution $(d^*, b^*)$ of 
the problem DR$'_{\gamma}$ by applying the steepest descent algorithm
in \cite[Section 6]{Shioura2022} (see also \cite[Section 4]{Shioura2022}).
 Using the optimal solution $(\hat{d}, \hat{b})$  
of the problem DR$'_{\gamma_{\rm min}}$ as an initial solution,
the steepest descent algorithm iteratively computes 
an optimal solution of the problem DR$'_{\gamma'}$ for 
$\gamma' = \gamma_{\rm min}+1, \gamma_{\rm min}+2, \ldots$
by local changes of a solution,
and finally finds 
an optimal solution of the problem DR$'_{\gamma}$, which
is also an optimal solution of DR$'$.
The running time of the steepest descent algorithm is given as 
$O(n +(\gamma - \gamma_{\rm min}) \log n)$ time.
The equation \eqref{eqn:gap-u'l'} implies 
$\gamma - \gamma_{\rm min} = O(n^2)$.
 Hence, the steepest descent algorithm
finds an optimal solution of DR$'$  in $O(n^2 \log n)$ time.

 This concludes the proof for the statement  that 
an optimal solution of DR$'(1)=$DR$'$ can be computed in $O(n^2 \log n)$ time
by using an optimal solution of DR$'(2)$.


\section{Proofs of Lemmas}

\subsection{Proof of Lemma \ref{lem:setI-4}}
\label{sec:apdx:1}

 All of the claims (i), (ii), (iii),  and (iv) can be proved in a similar way
as in Lemma \ref{lem:setI-2}.
 Below we give proofs of (i) and (ii) for the case $S=P$.

{[Proof of (i)]} \quad
 Assume, to the contrary, that
all of the sets $I_1$, $I^P_4$, and $I^P_5$ 
are nonempty, and let
$s \in I_1$, $i \in I^P_4$, and $j \in I^P_5$.
 Then, elements $s, i,j$ are distinct 
since sets $I_1, I^P_4$, and $I^P_5$ are mutually disjoint
(see Figures \ref{fig:I12}, \ref{fig:I34}, and \ref{fig:I56}).

 Consider the pair of vectors
$(d - \lambda\chi_s + \lambda \chi_i, b + \lambda\chi_s - \lambda \chi_j)$.
 Since $i \in I^P_4$ and $j \in I^P_5$, we have
\begin{align*}
&  (d(i) + \lambda) + b(i) = x(i)+ \lambda
 \le x^*(i) = d^*(i) + b^*(i) \le u'(i),
\\
& d(j) + (b(j) - \lambda) = x(j) - \lambda
 \ge x^*(j) =   d^*(j) + b^*(j) \ge \ell'(j),
\\
& (d(s) - \lambda)+ (b(s) + \lambda) = d(s)+b(s).
\end{align*}
 Since $i, j \in P$, we have
\begin{align*}
& (d - \lambda\chi_s+\lambda \chi_i)(P) +  (b + \lambda\chi_s - \lambda \chi_j)(P)
 =  d(P)+b(P), 
\\
& (d - \lambda\chi_s+\lambda \chi_i)(Q) +  (b + \lambda\chi_s - \lambda \chi_j)(Q)
 =  d(Q)+b(Q), 
\\
& (b + \lambda\chi_i - \lambda \chi_j)(N) = b(N) \le B.
\end{align*}
 From these inequalities and equations
we see that
$(d - \lambda\chi_s + \lambda \chi_i, b + \lambda\chi_s - \lambda \chi_j)$
 is a feasible solution of DR$'(\lambda)$.
 Since $(d,b)$ is an optimal solution of DR$'(\lambda)$,
 we have
\begin{equation}
c(d,b) \le
 c(d - \lambda\chi_s + \lambda \chi_i, b + \lambda\chi_s - \lambda \chi_j).
 \label{eqn:prox-1}
\end{equation}
 On the other hand,
it follows from the choice of $i,j,s$ that
\[
\min\{ d^*(i)-d(i), b(j)-b^*(j), 
d(s)-d^*(s), b^*(s)-b(s)\} \ge \lambda, 
\]
which, together with Lemma \ref{lem:mono-dec} (iv),
implies that 
\[
c(d,b) >
 c(d - \lambda\chi_s + \lambda \chi_i, b + \lambda\chi_s - \lambda \chi_j),
\]
a contradiction to \eqref{eqn:prox-1}.

{[Proof of (ii)]} \quad
 Assume, to the contrary, that
 $b(N) - b^*(N) \ge  \lambda$,
 $I^P_4 \ne \emptyset$, and $I^P_5 \ne \emptyset$ hold,
and let  $i \in I^P_4$ and $j \in I^P_5$.
 Then, elements $i,j$ are distinct 
since sets $I^P_4$ and $I^P_5$ are  disjoint.

 Consider the pair of vectors  $(d + \lambda \chi_i, b - \lambda \chi_j)$.
 Since $i,j \in P$, we have
\begin{align*}
& (d+\lambda \chi_i)(P) +  (b - \lambda \chi_j)(P)
 =  d(P)+b(P), 
\\
& (d + \lambda \chi_i)(Q) +  (b - \lambda \chi_j)(Q)
 =  d(Q)+b(Q), 
\\
& (b  - \lambda \chi_j)(N)  = b(N) - \lambda \le B,
\end{align*}
 Therefore, $(d + \lambda \chi_i, b - \lambda \chi_j)$
is a feasible solution of DR$'(\lambda)$.
 Since $(d,b)$ is an optimal solution of DR$'(\lambda)$,
 we have
\begin{equation}
  \label{eqn:prox-3}
c(d,b)  \le c(d + \lambda \chi_i, b - \lambda\chi_j).
\end{equation}
 On the other hand, it follows from  $i \in I^P_4$, $j \in I^P_5$,
and  $b(N) - b^*(N) \ge  \lambda$
that $\min\{ d^*(i)-d(i), b(j)-b^*(j), b(N)-b^*(N)\} \ge \lambda$,
which, together with Lemma \ref{lem:mono-dec} (iii),
implies that 
$c(d,b) > c(d + \lambda \chi_i, b - \lambda\chi_j)$,
a contradiction to \eqref{eqn:prox-3}.


\subsection{Proofs of the Bounds for the Remaining Cases}
\label{sec:apdx:3}


We first consider the cases P1--Q3, P1--Q4, P2--Q3, and P2--Q4.

\begin{figure}[t]
\begin{center}
{\includegraphics[width=0.6\textwidth]{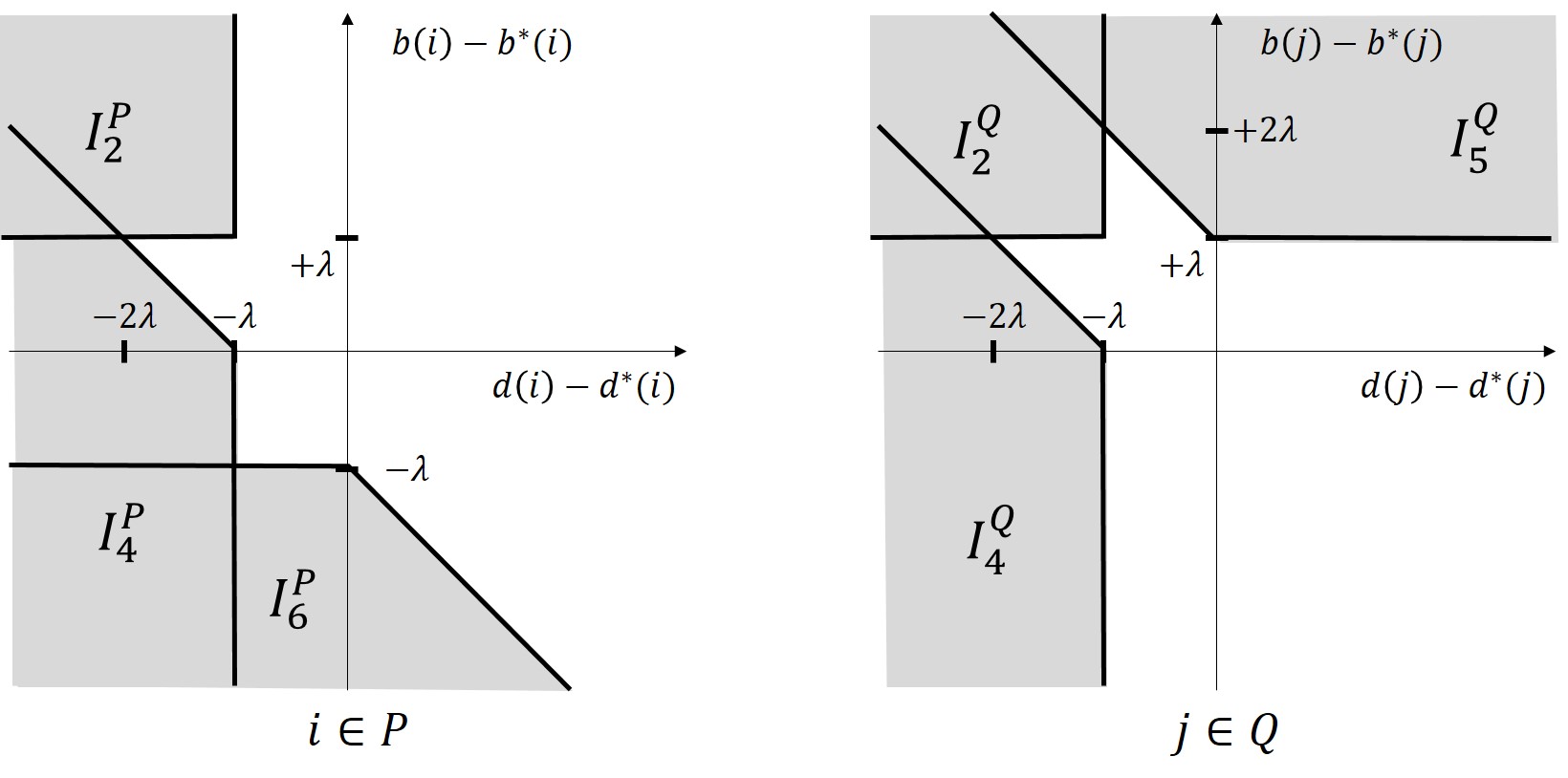}}
\caption{Regions corresponding to the case P1--Q3. 
Combination of the cases P1 and Q3 implies 
that the points $(d(i) - d^*(i), b(i) - b^*(i))$ for $i \in P$
and $(d(j) - d^*(j), b(j) - b^*(j))$ for $j \in Q$
are not located on the shaded regions.
\label{fig:P1Q3}
} 
\end{center}
{}
\end{figure}

\begin{lemma}
\label{lem:p1-q3-bound}
 We have $\| x-x^* \|_1 < 10 \lambda n$
if one of the cases P1--Q3, P1--Q4, P2--Q3, and P2--Q4 occurs.
\end{lemma}

\proof{Proof.}
 We first consider the case P1--Q3 (see Figure \ref{fig:P1Q3}).
 Since the case P1 occurs, we have
\[
 \sum_{i \in P}|x(i)- x^*(i)| < 4 \lambda |P| 
\]
(see \eqref{eqn:lem:p1-q2-bound:1} in the proof of Lemma \ref{lem:p1-q2-bound}).
 We will show that 
\begin{align}
& \|d- d^* \|_1 < 4\lambda n, \label{eqn:lem:p1-q3-bound:1}\\
& \sum_{i \in Q}|b(i)- b^*(i)| < 4 \lambda|Q| + 2 \lambda |P|. 
\label{eqn:lem:p1-q3-bound:2}
\end{align}
These inequalities imply the desired inequality as follows:
\begin{align*}
  \| x-x^* \|_1 
& = \sum_{i \in P}|x(i)- x^*(i)| + \sum_{i \in Q}|x(i)- x^*(i)|\\
& \le \sum_{i \in P}|x(i)- x^*(i)|
 + \sum_{i \in Q}|d(i)- d^*(i)| + \sum_{i \in Q}|b(i)- b^*(i)|\\
& \le \sum_{i \in P}|x(i)- x^*(i)|
 + \|d- d^*\|_1 + \sum_{i \in Q}|b(i)- b^*(i)|\\
& < 4 \lambda |P| + 4 \lambda n + (4 \lambda|Q| + 2 \lambda |P|)\\
& = 8 \lambda n  + 2 \lambda |P| \le 10 \lambda n.
\end{align*}

  We observe that (see Figure \ref{fig:P1Q3})
$I_4^P = I_6^P= I_2 = \emptyset$ implies
\begin{align}
 &
d(i)-d^*(i) \ge - 2\lambda, \quad
x(i) - x^*(i) \ge  - 2\lambda \qquad (i\in P),
\label{eqn:lem:p1-q3-bound:3}
\end{align} 
and
$I_4^Q= I_5^Q  = I_2 = \emptyset$ implies
\begin{align}
 &
d(i)-d^*(i) \ge - 2\lambda, \quad
b(i) - b^*(i) \le   2\lambda \qquad (i\in Q).
\label{eqn:lem:p1-q3-bound:4}
\end{align} 
 We use these inequalities to prove \eqref{eqn:lem:p1-q3-bound:1} 
and \eqref{eqn:lem:p1-q3-bound:2}.

{[Proof of \eqref{eqn:lem:p1-q3-bound:1}]} \quad
 Put $N' = \suppp(d-d^*)$.
 If $N' = \emptyset$, then we have
\begin{align*}
\|d - d^*\|_1
& = - \sum_{i \in N}(d(i)- d^*(i))
 <   2\lambda \cdot|N| = 4\lambda n,
\end{align*}
where the inequality is by
\eqref{eqn:lem:p1-q3-bound:3} and \eqref{eqn:lem:p1-q3-bound:4}.
 Otherwise, $|N \setminus N'| \le n-1$ and therefore it holds that
\begin{align*}
\|d - d^*\|_1
& = \sum_{i \in N'}(d(i)- d^*(i))
- \sum_{i \in N \setminus N'}(d(i)- d^*(i))\\
& =  (d(N')-d^*(N')) - (d(N\setminus N')-d^*(N\setminus N'))\\
& =  (d(N)-d^*(N)) - 2(d(N\setminus N')-d^*(N\setminus N'))\\
& < \lambda + 2 \cdot 2\lambda \cdot|N \setminus N'|\\
& \le 4\lambda n,
\end{align*}
where the first inequality is by the assumption 
 $d(N)-d^*(N) < \lambda$,
\eqref{eqn:lem:p1-q3-bound:3}, and \eqref{eqn:lem:p1-q3-bound:4}.

{[Proof of \eqref{eqn:lem:p1-q3-bound:2}]} \quad
 We first bound the values $b(P) - b^*(P)$ and $b(Q)- b^*(Q)$.
 It holds that
\begin{align*}
b(P)  - b^*(P)
& =\sum_{i \in P}(b(i)- b^*(i))\\
& =\sum_{i \in P}[(x(i)- x^*(i)) - (d(i)- d^*(i))]\\
& =(x(P)- x^*(P)) - \sum_{i \in P}(d(i)- d^*(i))
 \le 0 + 2\lambda |P|,  
\end{align*}
where the inequality is by 
\eqref{eqn:x-x*-ineq} and \eqref{eqn:lem:p1-q3-bound:3}.
 From this inequality it follows that 
\begin{align}
 b(Q)- b^*(Q)
& =  (b(N)- b^*(N)) - (b(P)- b^*(P))
 > - \lambda - 2\lambda |P|,
\label{eqn:p1-q3-ineq10}
\end{align}
where the inequality is by the assumption $b(N)- b^*(N) > - \lambda$.

 We now bound the value $\sum_{i \in Q}|b(i)- b^*(i)|$.
 Put $Q' = \suppp(b-b^*) \cap Q$.
 If $Q' = Q$, it holds that
\begin{align*}
\sum_{i \in Q}|b(i)- b^*(i)|
& = b(Q)- b^*(Q) \le  2\lambda |Q|,
\end{align*}
where the inequality is by \eqref{eqn:lem:p1-q3-bound:4}. 
 Otherwise, we have $|Q'| \le |Q| - 1$ and therefore
 it holds that
\begin{align*}
\sum_{i \in Q}|b(i)- b^*(i)|
& = (b(Q')- b^*(Q')) - (b(Q \setminus Q')- b^*(Q \setminus Q'))\\
& = 2(b(Q')- b^*(Q')) - (b(Q)- b^*(Q))\\
& < 2 \cdot 2\lambda |Q'| + (\lambda + 2 \lambda |P|) \\
& \le 4 \lambda|Q| + 2 \lambda |P|,
\end{align*}
where the first inequality is by \eqref{eqn:p1-q3-ineq10}.

This concludes the proof for the case P1--Q3.
We see that each of the cases P1--Q4, P2--Q3, and P2--Q4 is 
symmetric to the case P1--Q3 (see Figures \ref{fig:P1Q4}, \ref{fig:P2Q3},
and \ref{fig:P2Q4}),
the proofs for these cases can be done in the same way and therefore omitted.
\qed

\medskip

\begin{figure}[t]
\begin{center}
{\includegraphics[width=0.6\textwidth]{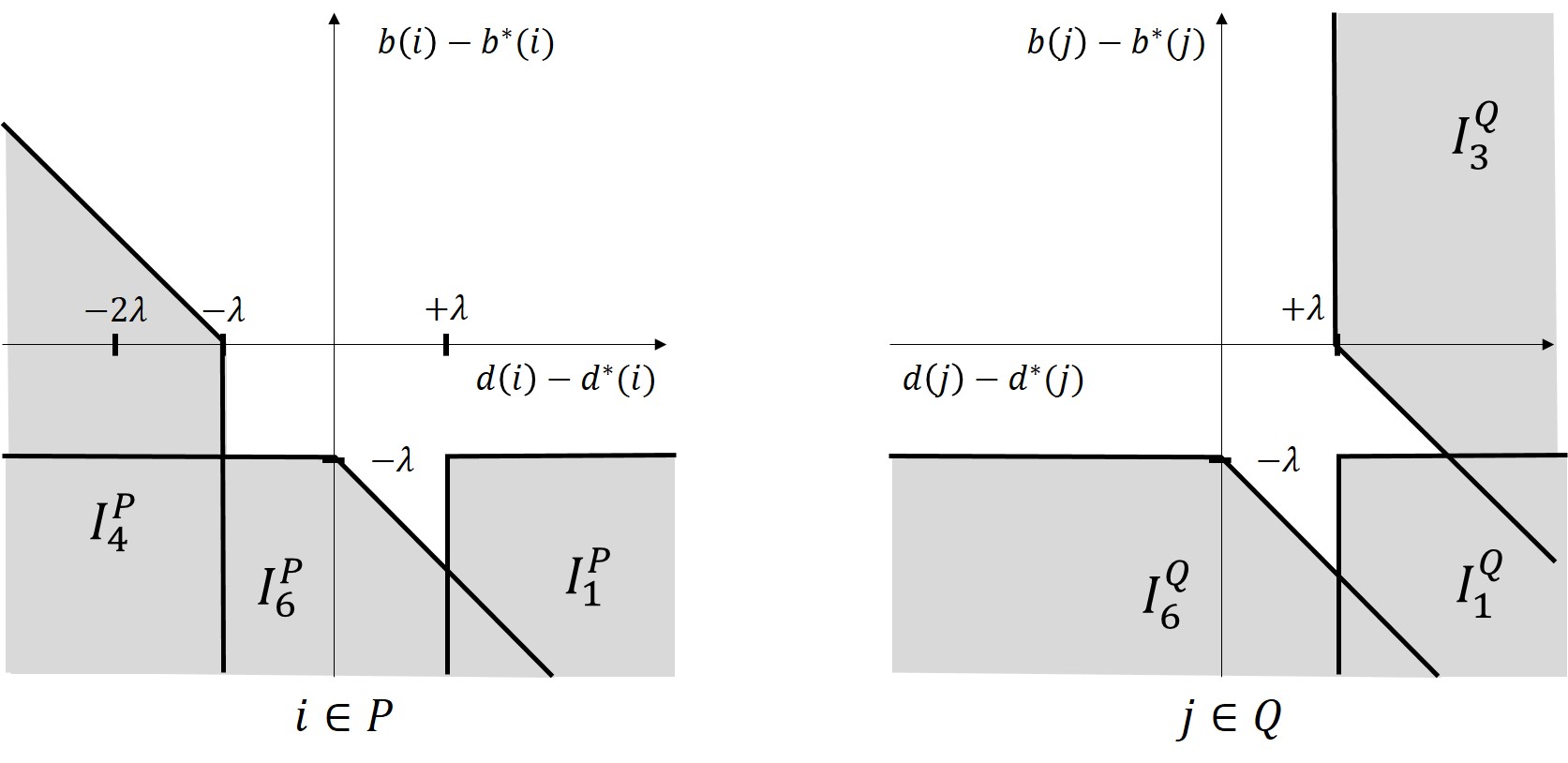}}
\caption{Regions corresponding to Case P1--Q4. 
\label{fig:P1Q4}
} 
\end{center}
{}
\end{figure}

\begin{figure}[t]
\begin{center}
{\includegraphics[width=0.6\textwidth]{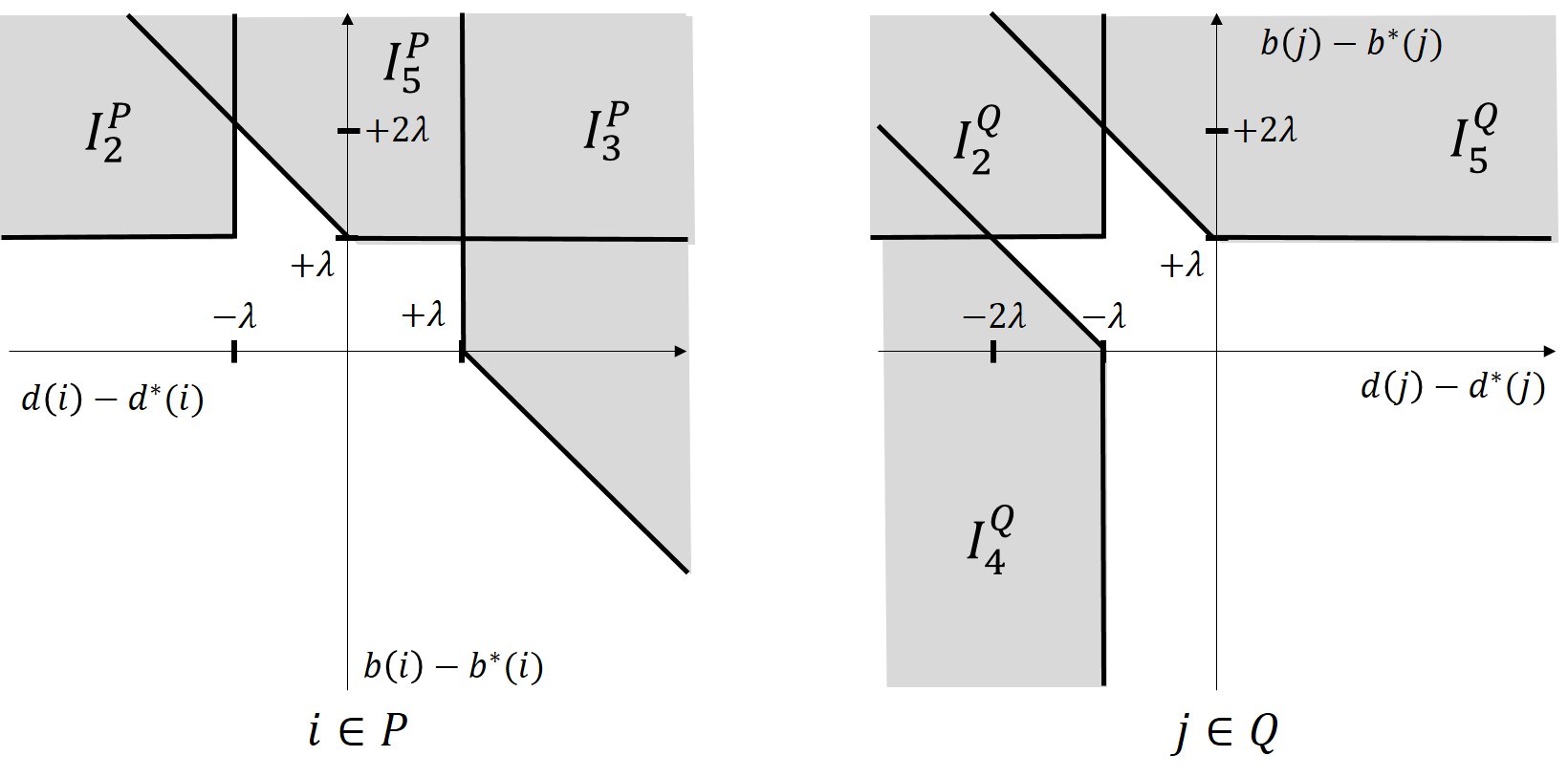}}
\end{center}
\caption{Regions corresponding to Case P2--Q3. 
\label{fig:P2Q3}
} 
{}
\end{figure}

\begin{figure}[t]
\begin{center}
{\includegraphics[width=0.6\textwidth]{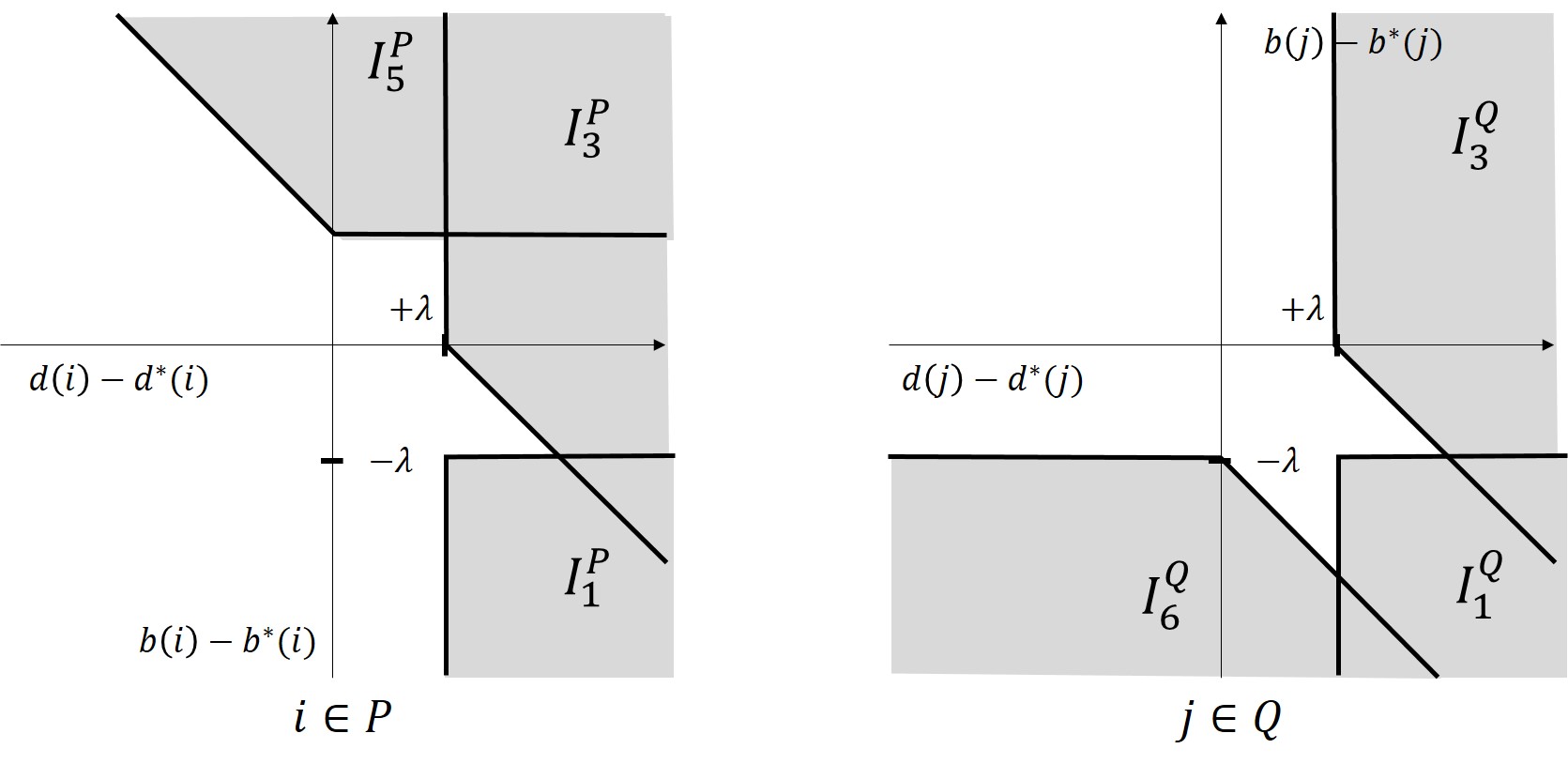}}
\caption{Regions corresponding to Case P2--Q4. 
\label{fig:P2Q4}
} 
\end{center}
\end{figure}

 The cases P3--Q1, P3--Q2, P4--Q1, and P4--Q2
are also symmetric to the cases P1--Q3, P1--Q4, P2--Q3, and P2--Q4,
and therefore we obtain the following bound for the cases:

\begin{lemma}
\label{lem:p3-q1-bound}
 We have $\| x-x^* \|_1 < 10 \lambda n$
if one of the cases P3--Q1, P3--Q2, P4--Q1, and P4--Q2 occurs.
\end{lemma}


\begin{thebibliography}{99}


\bibitem{FHS2022}
D. Freund, S.G.~Henderson, and  D.B.~Shmoys,  
Minimizing multimodular functions and
	allocating capacity in bike-sharing systems.
Oper. Res. 70 (2022), 2715--2731.


\bibitem{Hoch94}
D.S. Hochbaum,
Lower and upper bounds for the allocation problem and other nonlinear
optimization problems.
{Math. Oper. Res.} 19 (1994), 390--409.

\bibitem{HochShant90}
D.S. Hochbaum and J.G. Shanthikumar, Nonlinear separable optimization
is not much harder than linear optimization. J. ACM 37 (1990), 843--862.



\bibitem{MoriMuro2018}
S. Moriguchi and K. Murota,
On fundamental operations for multimodular functions.
{J. Oper. Res. Soc. Japan} 62 (2019), 53--63.



\bibitem{MST11}
S. Moriguchi, A. Shioura, and N. Tsuchimura,
M-convex function minimization by continuous relaxation approach:
proximity theorem and algorithm.
{SIAM J. Optim.} 21 (2011), 633--668.



 \bibitem{Murota03book}
K. Murota.
 \textit{Discrete Convex Analysis}. 
 SIAM, Philadelphia, 2003.



\bibitem{Murota05}
K. Murota,
Note on multimodularity and L-convexity. 
{Math. Oper. Res.} 30 (2005), 658--661.


\bibitem{MS99}
K. Murota and A. Shioura,
M-convex function on generalized polymatroid.
{Math. Oper. Res.} {24} (1999), 95--105.

\bibitem{Shioura04}
A. Shioura,
Fast scaling algorithms for M-convex function minimization 
with application to the resource allocation problem.
{Discrete Appl. Math.} 134 (2004), 303--316.



\bibitem{Shioura2022}
A. Shioura,
M-convex function minimization under L1-distance constraint
and its application to dock re-allocation in bike sharing system.
Math. Oper. Res. 47 (2022), 1566--1611.


\end{thebibliography}
\end{document}